\documentclass[12pt]{amsart}
\usepackage{amssymb,amsmath,amsthm,bm}
\usepackage[colorinlistoftodos,prependcaption,textsize=tiny]{todonotes}
 \definecolor{darkblue}{RGB}{0,0,160}
\usepackage{fouriernc} 
\usepackage[colorlinks=true,allcolors=darkblue]{hyperref}
\DeclareSymbolFont{usualmathcal}{OMS}{cmsy}{m}{n}
\DeclareSymbolFontAlphabet{\mathcal}{usualmathcal}

\usepackage[T1]{fontenc}
\usepackage{esint}
\usepackage{color}

\def\d{{\rm d}}

\def \and{\qquad\text{and}\qquad}

\newcounter{thms}
\newcounter{other}
\numberwithin{other}{section}

\newtheorem{theorem}[thms]{Theorem}
\newtheorem*{theorem*}{Theorem}
\newtheorem*{proposition*}{Proposition}
\newtheorem{corollary}{Corollary}
\newtheorem{lemma}[other]{Lemma}
\newtheorem{definition}[other]{Definition}
\newtheorem{remark}[other]{Remark}

\def\vint_#1{\mathchoice%
      {\mathop{\kern 0.2em\vrule width 0.6em height 0.69678ex depth -0.58065ex
              \kern -0.8em \intop}\nolimits_{\kern -0.4em#1}}%
      {\mathop{\kern 0.1em\vrule width 0.5em height 0.69678ex depth -0.60387ex
              \kern -0.6em \intop}\nolimits_{#1}}%
      {\mathop{\kern 0.1em\vrule width 0.5em height 0.69678ex depth -0.60387ex
              \kern -0.6em \intop}\nolimits_{#1}}%
      {\mathop{\kern 0.1em\vrule width 0.5em height 0.69678ex depth -0.60387ex
              \kern -0.6em \intop}\nolimits_{#1}}}
\def\vintslides_#1{\mathchoice%
      {\mathop{\kern 0.1em\vrule width 0.5em height 0.697ex depth -0.581ex
              \kern -0.6em \intop}\nolimits_{\kern -0.4em#1}}%
      {\mathop{\kern 0.1em\vrule width 0.3em height 0.697ex depth -0.604ex
              \kern -0.4em \intop}\nolimits_{#1}}%
      {\mathop{\kern 0.1em\vrule width 0.3em height 0.697ex depth -0.604ex
              \kern -0.4em \intop}\nolimits_{#1}}%
      {\mathop{\kern 0.1em\vrule width 0.3em height 0.697ex depth -0.604ex
              \kern -0.4em \intop}\nolimits_{#1}}}

\newcommand{\aveint}[2]{\mathchoice%
      {\mathop{\kern 0.2em\vrule width 0.6em height 0.69678ex depth -0.58065ex
              \kern -0.8em \intop}\nolimits_{\kern -0.45em#1}^{#2}}%
      {\mathop{\kern 0.1em\vrule width 0.5em height 0.69678ex depth -0.60387ex
              \kern -0.6em \intop}\nolimits_{#1}^{#2}}%
      {\mathop{\kern 0.1em\vrule width 0.5em height 0.69678ex depth -0.60387ex
              \kern -0.6em \intop}\nolimits_{#1}^{#2}}%
      {\mathop{\kern 0.1em\vrule width 0.5em height 0.69678ex depth -0.60387ex
              \kern -0.6em \intop}\nolimits_{#1}^{#2}}}

\renewcommand*{\cdots}{%
  \mathinner{{\cdotp}{\cdotp}{\cdotp}}%
}

\numberwithin{equation}{section}

\title[Sparse domination of Hilbert transforms along curves]{Sparse domination of Hilbert transforms along curves}

\author[L.\ Cladek]{Laura Cladek}
\address{\noindent Department of Mathematics, The University of British Columbia, \newline \indent 1984 Mathematics Road, Vancouver, BC V6T 1Z2 (L. Cladek)}
\email{cladek@math.ubc.ca}

\author[Y.\ Ou]{Yumeng Ou}
\address{\noindent Department of Mathematics, Massachusetts Institute of Technology, \newline \indent  77 Massachusetts Avenue, Cambridge, MA 02139, USA (Y. Ou)}
\email{yumengou@mit.edu} 

\thanks{The authors would like to thank Michael Lacey, Francesco Di Plinio and Ioannis Parissis for helpful discussions. This material is based upon work supported by the National Science Foundation under Grant No. DMS-1440140 while the authors were in residence at the Mathematical Sciences Research Institute in Berkeley, California, during the Spring 2017 semester.}

\begin{document}
\begin{abstract} 
We obtain sharp sparse bounds for Hilbert transforms along curves in $\mathbb{R}^n$, and derive as corollaries weighted norm inequalities for such operators. The curves that we consider include monomial curves and arbitrary $C^n$ curves with nonvanishing torsion.
\end{abstract}
\maketitle

\section{Introduction and main results}
In this article we prove sparse domination theorems for Hilbert transforms along certain classes of smooth curves in $\mathbb{R}^n$. The Hilbert transforms associated with these curves have a natural nonisotropic structure. As a corollary of our sparse domination theorems, we obtain weighted estimates for such operators where the weights belong to nonisotropic Muckenhoupt $A_p$ classes. As far as we are aware of, no prior results on weighted norm inequalities for singular integrals along curves appear to exist in the literature. The two main classes of curves we consider are monomial curves and curves with nonvanishing torsion.

In general, sparse bounds have been recognized as a finer quantification of the boundedness of singular integral operators such as Calder\'on-Zygmund operators (see for instance \cite{CDPO, Lacey1, LM, Lerner} and the references therein). Recently, a sparse domination principle was obtained for spherical maximal functions by Lacey in \cite{Lacey2} (see also the related work \cite{Oberlin} of Oberlin). Our article furthers this line of research and appears to be the first attempt to prove sparse bounds for Radon-type transforms when the underlying dilation structure is nonisotropic. 

We will prove that Hilbert transforms along certain classes of curves can be dominated in the bilinear form sense by positive sparse forms $\Lambda_{r,s}$ for pairs $(r,s)$ that are contained in the interior of the region where the associated single scale operator satisfies certain $L^p$ improving estimates. More precisely, given a curve $\gamma:\,\mathbb{R}\to \mathbb{R}^n$, there is a natural collection $\mathcal{Q}^\gamma$ of (nonisotropic) cubes associated to $\gamma$, which we refer to as \emph{$\gamma$-cubes} (see Definition \ref{gcubedef} and \ref{cubenvt} below). A sparse form is defined with respect to averages of functions on cubes of a \emph{sparse} collection. 
\begin{definition}\label{sparsity}
A collection of $\gamma$-cubes $\mathcal{S}\subset\mathcal{Q}^{\gamma}$ in $\mathbb{R}^n$ is \emph{$\delta$-sparse} for $0<\delta<1$ if there exist sets $\{E_S:\, S\in\mathcal{S}\}$ that are pairwise disjoint, $E_S\subset S$, and satisfy $|E_S|>\delta |S|$ for all $S\in\mathcal{S}$. For any $\gamma$-cube $Q$, $1\leq p<\infty$, define $\langle f\rangle_{Q,p}^p:=|Q|^{-1}\int_Q|f|^p\,\d x$. Then the $(r,s)$-sparse form $\Lambda_{\mathcal{S},r,s}$ is defined as
\[
\Lambda_{\mathcal{S},r,s}(f,g):=\sum_{S\in\mathcal{S}}|S|\langle f\rangle_{S,r}\langle g\rangle_{S,s},
\]
for $1\leq r,s<\infty$. We sometimes abbreviate this as $\Lambda_{r,s}$ when there is no need to specify the sparse collection $\mathcal{S}$.
\end{definition}

One of the main contributions of the present work is to establish a link between the geometry of the curve and the extent to which the associated singular integral operator can be sparse dominated, described by the range of the pairs $(r,s)$ where a sparse form $\Lambda_{r,s}$ domination is available. Our theorems are sharp up to the endpoint. That is, we are able to obtain a region in the $(1/r, 1/s)$ plane such that a sparse domination result holds for pairs $(r, s)$ in the interior of the region and fails for pairs in the exterior of the region. We refer to Section \ref{weighted} for a proof of this fact.

The main ideas of the proof of the sparse domination theorems (Theorem \ref{main} and \ref{main2} below) are inspired by the argument of Lacey in \cite{Lacey2}. The proof involves several main ingredients including $L^p$ improving estimates for the single scale averaging operator associated with curves with nonvanishing torsion, Fourier decay estimates for the arclength measure along such curves which allows one to obtain certain translation-continuity $L^p$-improving estimates for each single scale averaging operator, and uniform $L^p$ bounds for truncations of the full (multi-scale) operator.

\subsection{Monomial curves}

We begin by considering monomial curves. Let $\gamma: \mathbb{R}\to \mathbb{R}^n$ be a monomial curve, that is, 
\[
\gamma(t)=\begin{cases} (\epsilon_1|t|^{\alpha_1},\ldots,\epsilon_n|t|^{\alpha_n}) & \text{if}\,t\geq 0 \\
(\epsilon'_1 |t|^{\alpha_1},\ldots,\epsilon'_n|t|^{\alpha_n}) & \text{if}\,t<0
\end{cases}
\] for real numbers $0<\alpha_1<\cdots<\alpha_n<\infty$ where
\[
\vec{\epsilon}, \vec{\epsilon}'\in \{1,-1\}^n
\]and there exists $j$ so that $\epsilon_j\neq \epsilon'_j$. Define the Hilbert transform along $\gamma$ as
\[
H_{\gamma}f(x):=p.v.\int_{\mathbb{R}}f(x-\gamma(t))\frac{\d t}{t},\quad x\in\mathbb{R}^n.
\]It is obtained in \cite{SW} that $H_\gamma$ together with its maximal truncation are bounded on $L^p(\mathbb{R}^n)$, $1<p<\infty$.

For $\lambda>0$, We may define single scale operators $A_\lambda^{\gamma}$ by
\begin{equation}\label{average}
A^{\gamma}_{\lambda} f(x):=\int_{\frac{\lambda}{2}\leq |t|<\lambda} f(x-\gamma(t))\,\frac{\d t}{t},
\end{equation}so that in the distributional sense, we have
\[H_\gamma f=\sum_{k\in\mathbb{Z}}A^{\gamma}_{2^k}f.\]
In other words, $H_\gamma$ can be decomposed into a sum of single scale operators at different scales. There is a natural family of nonisotropic dilations $\left\{\delta^{\gamma}_{\lambda}\right\}_{\lambda>0}$ associated with $\gamma$, given by
\[\delta^\gamma_\lambda(x):=(\lambda^{\alpha_1}x_1,\ldots,\lambda^{\alpha_n}x_n).
\]
Observe that for any $\lambda>0$ we have the relation
\[
A^\gamma_{\lambda}f(x)=A^\gamma_1(f\circ\delta^\gamma_\lambda(\cdot))(\delta^\gamma_{\lambda^{-1}}(x)).
\]

We will work with generalized cubes that reflect the geometry of the nonisotropic dilation group associated with $\gamma$.
\begin{definition}\label{gcubedef}
A \emph{$\gamma$-cube} is a hyperrectangle in $\mathbb{R}^n$ with sides parallel to the coordinate axes whose side-lengths $(\ell_1,\ldots,\ell_n)$ satisfy the relation $\ell_1^{1/\alpha_1}=\cdots=\ell_n^{1/\alpha_n}$. We denote the collection of all $\gamma$-cubes by $\mathcal{Q}^{\gamma}$. 
\end{definition}
For $Q\in\mathcal{Q}^{\gamma}$, we write $$\vec\ell(Q)=(\ell_1(Q),\ldots,\ell_n(Q)).$$ For a fixed $\gamma$, any $\ell_j(Q)$ dictates all of the rest of the side-lengths. For our convenience, we introduce another notation $\ell(Q)$ defined as
\[
\ell(Q):=\ell_1(Q)^{1/\alpha_1}=\cdots=\ell_n(Q)^{1/\alpha_n}.
\]In particular, $|Q|=\ell(Q)^{\sum_{j=1}^n\alpha_j}$.
\begin{remark}
Equivalently, up to translation, $\mathcal{Q}^{\gamma}$ is the orbit of $[0,1]^n$ under the dilation group $\delta^\gamma$. For example, when $\gamma(t)=(t,t^2)$ is the parabola in $\mathbb{R}^2$, $\mathcal{Q}^{\gamma}$ is the collection of all parabolic cubes $Q$ whose side-lengths are of the form $(\ell, \ell^2)$. \end{remark}

Now we are ready to state our first theorem.

\begin{theorem}\label{main}
Given a monomial curve $\gamma$ in $\mathbb{R}^n$, $n\geq 2$, and $(r, s)$ so that $(1/r,1/s')$ lies in the interior of the trapezoid $\Omega$ with vertices
\begin{align*}
(0, 0),~ (1, 1), ~\bigg(\frac{n^2-n+2}{n^2+n}, \frac{n-1}{n+1}\bigg),~\bigg(\frac{2}{n+1}, \frac{2n-2}{n^2+n}\bigg).
\end{align*}
For any compactly supported bounded functions $f,g$ on $\mathbb{R}^n$, and any $\delta=\delta(\gamma)>0$ there exists a $\delta$-sparse collection $\mathcal{S}$ of $\gamma$-cubes such that
\[
|\langle H_\gamma f,g\rangle|\leq C(n,\gamma)\Lambda_{\mathcal{S},r,s}(f,g).
\]
\end{theorem}

The \emph{sparse region} for admissible pairs $(r,s)$ in the above theorem is the trapezoid $\Omega'$ with vertices
\begin{align*}
(0, 1),~ (1, 0), ~\bigg(\frac{n^2-n+2}{n^2+n}, \frac{2}{n+1}\bigg),~\bigg(\frac{2}{n+1}, \frac{n^2-n+2}{n^2+n}\bigg),
\end{align*}which is symmetric about $1/s=1/r$.

\begin{remark}The sparse collection $\mathcal{S}$ obtained in Theorem \ref{main} can be constructed to contain only dyadic versions of $\gamma$-cubes. However, some care is needed with the definition of dyadic $\gamma$-cubes. If a given curve $\gamma$ is defined via some irrational $\alpha_j$, even when $\ell(Q)=2^k$, $k\in\mathbb{Z}$, the $\gamma$-cube $Q$ will not necessarily have dyadic side-lengths. Thus we define $Q$ to be a \emph{dyadic $\gamma$-cube} if its side-lengths are of the form $(2^{\lfloor k\alpha_1\rfloor}, \ldots, 2^{\lfloor k\alpha_n\rfloor})$ for some $k\in\mathbb{Z}$, and denote $\ell(Q):=2^k$ for the largest possible $k$. Note that a dyadic $\gamma$-cube may not be a $\gamma$-cube, however, they are essentially equivalent. More precisely, given any $\gamma$-cube $Q$, there exists a dyadic $\gamma$-cube $\widetilde{Q}\subset Q$ such that $\ell(Q)<2\ell(\widetilde{Q})$, and vice versa. In fact, we will first prove that Theorem \ref{main} holds true for sparse forms defined via dyadic $\gamma$-cubes, then pass to the desired claim.
\end{remark}

\subsection{Curves with nonvanishing torsion}
Let $\gamma: [-1, 1]\to\mathbb{R}^n$ be a $C^n$ curve with \emph{non-vanishing torsion}; i.e. satisfying for all $s\in (-1, 1)$,
$$\det\begin{pmatrix} \gamma^{(1)}(s) & \cdots &\gamma^{(n)}(s)\end{pmatrix}\ne 0.$$
For curves with nonvanishing torsion, one must consider a local operator (i.e., restrict the integration in $t$ in the definition of $H_{\gamma}$ to the interval $[-1, 1]$). This is necessary to ensure that one may indeed associate natural nonisotropic cubes with the operator associated with a family of balls belonging to a space of homogeneous type. While the curve itself lacks a true nonisotropic dilation structure, it is very nearly endowed with such structure, which allows one to define the appropriate nonisotropic cubes. Thus we define the Hilbert transform along $\gamma$ as
\begin{equation}\label{Hnvt}
H_{\gamma}f(x):=p.v.\int_{-1}^1f(x-\gamma(t))\frac{\,dt}{t}.\end{equation}
Although $\gamma$ is not generated by a family of nonistropic dilations, we may still associate a family of cubes to $\gamma$, dictated by the Taylor expansion of $\gamma$ at the origin. 

To begin with, we assume $\gamma(0)=0$ without loss of generality. At each $0<|t|<1$, one has the Taylor expansion in each component, given by
\begin{equation}\label{Taylor}
\gamma(t)=\gamma'(0)t+\frac{\gamma^{(2)}(0)}{2}t^2+\cdots +\frac{\gamma^{(n)}(0)}{n!}t^n+O(t^{n+1}).
\end{equation}The nonvanishing torsion condition then guarantees that the vectors $\gamma'(0), \gamma''(0),\ldots\gamma^{(n)}(0)$ are linearly independent, i.e. they span the vector space $\mathbb{R}^n$. Therefore, one can accordingly define a new coordinate system such that $\gamma$ lies inside some axes-parallel cubes at each scale, similarly as the monomial curve case.

More precisely, let $e_1$ be the unit vector parallel to $\gamma'(0)$, and for each $1<k\leq n$, let $e_k$ be the unit vector parallel to $\gamma^{(k)}(0)-\text{Proj}_{\text{span}(e_1, \ldots, e_{k-1})}\gamma^{(k)}(0)$. Then $\{e_1,\ldots,e_n\}$ generates a coordinate system which is a rotation of the standard coordinate system. 

\begin{definition}\label{cubenvt}
For such a $\gamma$ with non-vanishing torsion, we define a \textit{$\gamma$-cube} $Q$ of side-length $\ell$ (denoted by $\ell(Q)$) to be a hyperrectangle in $\mathbb{R}^n$ with sides parallel to the vectors $e_1, e_2, \ldots, e_n$ so that the side-length of the side parallel to $e_k$ is $\ell^k$, $\forall 1\leq k\leq n$.
\end{definition}

\begin{remark}
Since we are considering a local operator, the remainder term $O(t^{n+1})$ is always negligible, which is why the curve $\gamma$ can be well approximated by the cubes defined above. 
\end{remark}

Observe that the $\gamma$-cubes defined above can be viewed as rotations of the cubes studied in the previous subsection associated to the moment curve $(t,t^2,\ldots,t^n)$, and their orientation (at all scales) is determined by the derivatives of $\gamma$ at $0$ up to order $n$. 

\begin{theorem}\label{main2}
Given a $C^n$ curve $\gamma: [-1, 1]\to\mathbb{R}^n$ with nonvanishing torsion, $n\geq 2$. Let $(r, s)$ be so that $(1/r,1/s')$ lies in the interior of the trapezoid $\Omega$ as defined in the statement of Theorem \ref{main}, for any compactly supported bounded functions $f,g$ on $\mathbb{R}^n$, there exists a sparse collection $\mathcal{S}$ of $\gamma$-cubes such that
\[
|\langle H_\gamma f,g\rangle|\leq C(n,\gamma)\Lambda_{\mathcal{S},r,s}(f,g).
\]
\end{theorem}

In fact, one can construct the sparse collection $\mathcal{S}$ in Theorem \ref{main2} so that $\forall Q\in\mathcal{S}$, there holds $\ell(Q)\lesssim 1$: see Remark \ref{small cube} of Section \ref{nonvanishing} for a more detailed discussion.

In the rest of the article, we prove Theorem \ref{main} and Theorem \ref{main2} in Section \ref{mono} and \ref{nonvanishing} respectively. The sharp weighted norm inequalities implied by the sparse domination theorems will be discussed in Section \ref{weighted}, where a proof of the sharpness of the sparse region $\Omega$ is also presented.

\section{Monomial curves: proof of Theorem \ref{main}}\label{mono}
Given $y\in\mathbb{R}^n$, let $\tau_yf(x):=f(x-y)$ be the translation operator by $y$. A key ingredient of our argument is the following $L^p$ improving continuity estimate for the single scale operators.
\begin{lemma}\label{cty1}
Let $(r, s)$ be so that $(1/r,1/s)$ lies in the interior of the region $\Omega$, where $\Omega$ is the trapezoid as defined in the statement of Theorem \ref{main}. Then for any $y\in\mathbb{R}^n$ satisfying $|y_j|\leq 1, \forall j=1,\ldots,n$,
\[
\|A_1^\gamma-\tau_y A_1^\gamma: L^r\to L^s\|\lesssim |y|^\eta
\]for some $\eta=\eta(r,s,n)>0$.
\end{lemma}
\begin{proof}
For any $(r,s)$ in the claimed range, $A_1^{\gamma}$ is dominated by an averaging operator corresponding to a curve with nonvanishing torsion, since the condition $0<\alpha_1<\cdots<\alpha_n$ guarantees that $\gamma$ has nonvanishing torsion away from $t=0$. For $(r, s)$ in the range described above, the following non-endpoint $L^r\to L^s$ improving estimate is proved in \cite{Christ1}, \cite{TW}:
\[
\|A_1^\gamma:\,L^r\to L^s\|\lesssim 1.
\]By an easy application of van der Corput's lemma, one may show that if $\mu:=\psi(t)\nu(t)\,dt$ for $\psi$ a smooth normalized bump and $\nu$ a curve with nonvanishing torsion on the support of $\psi$, then $|\widehat{\mu}(|\xi|)|\lesssim (1+|\xi|)^{-1/n}$. Thus by Plancherel, one also has
\[
\|A_1^\gamma-\tau_y A_1^\gamma: L^2\to L^2\|\lesssim\|(1-e^{i y\cdot \xi})\widehat{d\sigma_\gamma}(\xi)\|_{\infty}\lesssim |y|^{\eta_0},\quad \eta_0=\eta_0(n)>0.
\]The desired estimate in the entire region of $(r,s)$ then follows from interpolating between the two estimates above.
\end{proof}
By change of variables, this implies immediately a scale invariant version:
\[
\|A_\lambda^\gamma-\tau_yA_{\lambda}^\gamma:\, L^r\to L^s\|\lesssim \left|\left(\frac{y_1}{\lambda^{\alpha_1}},\ldots,\frac{y_n}{\lambda^{\alpha_n}}\right)\right|^\eta\lambda^{(1/s-1/r)\sum_{j=1}^n\alpha_j}
\]whenever $|y_j|\leq \lambda^{\alpha_j}$, $\forall j=1,\ldots,n$. This then implies the following lemma.
\begin{lemma}\label{ctyt}
Let $f_1,f_2$ be supported on a $\gamma$-cube $Q$ and let $\lambda\sim \ell(Q)$. For any $y\in\mathbb{R}^n$ such that $|y_j|\leq\ell(Q)^{\alpha_j}$, $\forall j=1,\ldots,n$, and $(r,s)$ be such that $(1/r, 1/s')$ lies in the interior of $\Omega$ as in the statement of Theorem \ref{main}, there holds
\[
|\langle A_{\lambda}^\gamma f_1-\tau_y A_{\lambda}^\gamma f_1,f_2\rangle|\lesssim \left|\left(\frac{y_1}{\ell(Q)^{\alpha_1}},\ldots,\frac{y_n}{\ell(Q)^{\alpha_n}}\right)\right|^\eta |Q|\langle f_1\rangle_{Q,r}\langle f_2\rangle_{Q,s},
\]where the implicit constant depends on $r,s,\gamma$.
\end{lemma}

We start the proof of Theorem \ref{main} with decomposing $A_{2^k}^\gamma$ into localized pieces. For the sake of brevity, we will drop the dependence on $\gamma$ from now on. {{For a dyadic cube $Q$ with $\ell(Q)=2^q$, $q\in\mathbb{Z}$, define 
$$A_Qf=A_{2^{q-N}}(f\chi_{\frac{1}{3}Q})$$ 
for some sufficiently large constant $N=N(\gamma)$ to be determined later. Here $cQ$ denotes the cube with the same center as $Q$ such that its side-length in the $k$-th direction is $c\ell(Q)^{\alpha_k}$, $k=1,\ldots,n$. Note that $cQ$ is not necessarily a $\gamma$-cube. 

The indicator function on $\frac{1}{3}Q$ is inserted to make sure that $A_Qf$ is supported in $Q$. Indeed, for any $x\in\text{spt}\,(A_Qf)$, the distance between $x$ and the center of $Q$ in the $k$-th direction satisfies
\[
|x_k-c(Q)_k|\leq 2^{(q-N)\alpha_k}+2^{-1}3^{-1}2^{q\alpha_k}=2^{-1}2^{q\alpha_k}(2^{1-N\alpha_k}+3^{-1})<1,
\]provided that $N$ is chosen sufficiently large depending on $\{\alpha_1,\ldots,\alpha_n\}$.}}

One can define the associated $3^n$ universal shifted dyadic grids $\mathcal{D}^\gamma_{\vec{j}}$ as
\begin{equation}\label{grids}
\left\{2^{\lfloor k\alpha_1\rfloor}\left[m_1+\frac{j_1}{3},m_1+1+\frac{j_1}{3}\right]\times\cdots\times 2^{\lfloor k\alpha_n\rfloor}\left[m_n+\frac{j_n}{3},m_n+1+\frac{j_n}{3}\right],\,k\in\mathbb{Z}, \vec{m}\in\mathbb{Z}^n\right\}
\end{equation}where $\vec{j}\in\{0,1,2\}^n$. These grids behave very nicely as they are dyadic grids in a space of homogeneous type. In particular, let $\mathcal{D}$ be a generic dyadic grid of the above type, then it preserves the following key properties of a standard (isotropic) dyadic grid:
\begin{enumerate}
\item $\mathcal{D}=\bigcup_k\mathcal{D}_k$ and each $\mathcal{D}_k$ partitions $\mathbb{R}^n$;
\item (Nesting property) For $Q_1,Q_2\in\mathcal{D}$ such that $\ell(Q_1)\leq\ell(Q_2)$, either $Q_1\subset Q_2$ or $Q_1\cap Q_2=\emptyset$;
\item Given $Q\in\mathcal{D}$, let $Q^{(1)}\in\mathcal{D}$ be the smallest cube such that $Q\subsetneq Q^{(1)}$, then $\ell(Q^{(1)})\leq C\ell(Q)$ for $C=C(n,\gamma)$;
\item Differentiation property: $\forall x\in\mathbb{R}^n$, there exists a chain $\{Q_i\}\subset\mathcal{D}$ containing $x$ such that $\lim_{i\to\infty}\ell(Q_i)=0$, i.e. $\{Q_i\}$ converges to the point $x$. 
\end{enumerate}
Note that unlike the standard case, a dyadic cube $Q\in\mathcal{D}$ can belong to more than one generations $\mathcal{D}_k$ because of the way we define dyadic cubes, which will however be irrelevant to our argument. {{Similarly as in the standard case, at each fixed scale $q$, 
\[\left\{\frac{1}{3}Q:\, \ell(Q)=2^q, Q\in\mathcal{D}_{\vec{j}}^\gamma,\,\text{for some}\,\vec{j}\right\}\]forms a disjoint partition of $\mathbb{R}^n$, hence gives rise to the following decomposition:
\[
A_{2^{q-N}}f=\sum_{\vec{j}}\sum_{\substack{Q:\,Q\in\mathcal{D}_{\vec{j}}\\ \ell(Q)=2^{q}}} A_Qf,
\]which leads to the decomposition
\[
H_\gamma= \sum_{\vec{j}}\sum_{Q\in\mathcal{D}_{\vec{j}}}A_Qf.
\]}}Since there are only finitely many dyadic grids in the above formula and our estimate doesn't depend on the specific grid, it suffices to study the operator $\sum_{Q\in\mathcal{D}}A_Qf$ for a fixed dyadic grid $\mathcal{D}$.

We claim that Theorem \ref{main} follows from iterating the following lemma.
\begin{lemma}\label{iteration}
Let $(r,s)$ be as in the statement of Theorem \ref{main}, and $C_0>1$ be a large enough constant. Fix a dyadic $\gamma$-cube $Q_0$. Let $\mathcal{Q}$ be a collection of sub dyadic $\gamma$-cubes of $Q_0$ so that
\begin{equation}\label{lemcond}
\sup_{Q\in\mathcal{Q}}\sup_{Q':\,Q\subset Q'\subset Q_0}\left|\frac{\langle f_1\rangle_{Q',r}}{\langle f_1\rangle_{Q_0,r}}+\frac{\langle f_2\rangle_{Q',s}}{\langle f_2\rangle_{Q_0,s}}\right|<C_0.
\end{equation}Then there holds
\[
\left\langle \sum_{Q\in\mathcal{Q}}A_Qf_1,f_2\right\rangle\lesssim |Q_0|\langle f_1\rangle_{Q_0,r}\langle f_2\rangle_{Q_0,s}.
\]
\end{lemma}

Indeed, Lemma \ref{iteration} implying Theorem \ref{main} can be seen from a by now standard sparse domination argument (for example in \cite{Lacey2}). The only slight differences here are that our cubes are nonisotropic, and one needs to pass from dyadic cubes back to $\gamma$-cubes. We include the deduction for the sake of being self-contained. Set $\mathcal{S'}=\emptyset$. Suppose there exists a dyadic $\gamma$-cube $Q_0$ such that $f_1,f_2$ are both supported in $Q_0$ and all terms in the bilinear form $\sum_{Q}\langle A_Qf_1,f_2\rangle$ are zero unless $Q\subset Q_0$. (This can be achieved by assuming the series to be finite and divide $\mathbb{R}^n$ into different quadrants. Our estimate will be uniform over all finite series.) We then start with adding $Q_0$ to $\mathcal{S'}$. Next, define $\mathcal{E}$ as the collection of maximal dyadic $\gamma$-cubes $P\subsetneq Q_0$ such that $\langle f_1\rangle_{P,r}>C\langle f_1\rangle_{Q_0,r}$ or $\langle f_2\rangle_{P,s}>C\langle f_2\rangle_{Q_0,s}$. $\mathcal{E}$ is well defined since the dyadic grid consists of nested cubes, which also implies that the cubes $P\in\mathcal{E}$ are pairwise disjoint and there holds
\[
\left|\bigcup_{P\in\mathcal{E}}P\right|<\frac{1}{2}|Q_0|
\]when $C$ is chosen sufficiently large. Let $\mathcal{Q}:=\big\{Q\subset Q_0:\,Q\not\subset \bigcup_{P\in\mathcal{E}}P\big\}$, then
\[
\sum_{Q\subset Q_0}\langle A_Qf_1,f_2\rangle=\sum_{Q\in\mathcal{Q}}\langle A_Qf_1,f_2\rangle+\sum_{P\in\mathcal{E}}\sum_{Q\subset P}\langle A_Qf_1,f_2\rangle.
\]By Lemma \ref{iteration}, the first term above is bounded by $|Q_0|\langle f_1\rangle_{Q_0,r}\langle f_2\rangle_{Q_0,s}$, and the second term will go into iteration. We then add $\mathcal{E}$ into $\mathcal{S'}$ and recurse the argument with $Q_0$ replaced by each $P\in\mathcal{E}$. The recursion stops when all cubes that appear in the series are exhausted, and we have a sparse form $\Lambda_{\mathcal{S'},r,s}(f_1,f_2)$ associated with a sparse collection of dyadic $\gamma$-cubes. One then observes that $\forall S'\in\mathcal{S'}$, there exists a $\gamma$-cube $S\supset S'$ with $\ell(S)<2\ell(S')$. Let $\mathcal{S}$ be the collection of these $\gamma$-cubes $S$, which is obviously still $\frac{1}{2}$-sparse and there holds
\[
\left\langle \sum_{Q\in \mathcal{D}}A_Qf_1,f_2\right\rangle\leq C'\Lambda_{\mathcal{S}',r,s}(f_1,f_2)\leq C\Lambda_{\mathcal{S},r,s}(f_1,f_2).
\]The proof of Theorem \ref{main} is thus complete up to the verification of Lemma \ref{iteration}.

\begin{proof}[Proof of Lemma \ref{iteration}]
Without loss of generality, assume both $f_1,f_2$ are non-negative. By homogeneity one can also assume that $\langle f_1\rangle_{Q_0,r}=\langle f_2\rangle_{Q_0,s}=C_0$. Perform a Calder\'on-Zygmund decomposition of $f_1,f_2$ at level $C_0$. More precisely, let $\mathcal{L}$ be the collection of maximal dyadic sub $\gamma$-cubes $L$ of $Q_0$ so that
\[
\langle f_1\rangle_{L,r}+\langle f_2\rangle_{L,s}>2C_0,
\]then write
\[
f_i=g_i+b_i,\quad b_i:=\sum_{L\in\mathcal{L}}b_{i,L}:=\sum_{L\in\mathcal{L}}\left(f_i-\langle f_i\rangle_L\right)\chi_L=:\sum_{k=-\infty}^{q_0-1}b_{i,k},
\]where
\[
b_{i,k}:=\sum_{L\in\mathcal{L}:\, \ell(L)=2^k}b_{i,L}.
\]

We then decompose
\[
\left\langle \sum_{Q\in\mathcal{Q}}A_Qf_1,f_2\right\rangle=\left\langle \sum_{Q\in\mathcal{Q}}A_Qg_1,g_2\right\rangle+\left\langle \sum_{Q\in\mathcal{Q}}A_Qg_1,b_2\right\rangle+\left\langle \sum_{Q\in\mathcal{Q}}A_Qb_1,g_2\right\rangle+\left\langle \sum_{Q\in\mathcal{Q}}A_Qb_1,b_2\right\rangle.
\]We claim that the first three terms above are bounded by $|Q_0|$ as a consequence of the uniform boundedness of truncations of $H_\gamma$ on $L^p$, $\forall 1<p<\infty$. 

Indeed, observing that $\|g_i\|_{L^\infty}\lesssim 1$, which follows from the fact that our dyadic grid has the differentiation property and $|Q^{(1)}|\leq C(n,\gamma)|Q|$, $\forall Q\in\mathcal{D}$,
\[
\left|\left\langle \sum_{Q\in\mathcal{Q}}A_Qg_1,g_2\right\rangle\right|\lesssim \|g_1\|_{L^2}\|g_2\|_{L^2}\leq\|g_1\|_{L^\infty}\|g_2\|_{L^\infty}|Q_0|\lesssim |Q_0|.
\]The two mixed terms follow similarly by virtue of the fact that $\|b_1\|_{L^r}\lesssim |Q_0|^{1/r}$ and $\|b_2\|_{L^s}\lesssim |Q_0|^{1/s}$, which is because $\|b_{1,L}\|_{L^r}\lesssim |L|^{1/r}$, $\|b_{2,L}\|_{L^s}\lesssim |L|^{1/s}$, and $\left\{L\in\mathcal{L}\right\}$ are pairwise disjoint stopping cubes. Take the second term as an example, one has
\[
\left|\left\langle \sum_{Q\in\mathcal{Q}}A_Qg_1,b_2\right\rangle\right|\lesssim \|g_1\|_{L^{s'}}\|b_2\|_{L^s}\lesssim \|g_1\|_{L^\infty}|Q_0|^{1/s'}|Q_0|^{1/s}\lesssim |Q_0|.
\]

We are thus left with the last term where bad functions appear in both components. It can be estimated in a similar way as in the proof of Lemma 2.6 of \cite{Lacey2}, with the added difficulty that one has to analyze the interaction between $\gamma$-cubes and dyadic $\gamma$-cubes in order to make use of Lemma \ref{ctyt}. We also emphasize that this ``bad-bad'' case is the only place in the entire proof of Theorem \ref{main} where the $L^p$ improving continuity estimate shown in Lemma \ref{ctyt} comes into play. 

By the assumption (\ref{lemcond}) and the expansion of the bad functions,
\[
\left|\left\langle \sum_{Q\in\mathcal{Q}}A_Qb_1,b_2\right\rangle\right|\leq\sum_{k_1=1}^\infty\sum_{k_2=1}^\infty\sum_{Q\in\mathcal{Q}}\left|\left\langle A_Q b_{1,q-k_1},b_{2,q-k_2}\right\rangle\right|,
\]where we have denoted $2^q:=\ell(Q)$. Further splitting into two cases according to whether $k_1\leq k_2$ or $k_1>k_2$, we claim that
\begin{equation}\label{lemclaim}
\left|\left\langle A_Qb_{1,q-k_1},b_{2,q-k_2}\right\rangle\right|\lesssim 2^{-\alpha\eta(k_1+k_2)/2}|Q|\langle b_{1,q-k_1}\rangle_{Q,r}\langle b_{2,q-k_2}\rangle_{Q,s},
\end{equation}where $\alpha:=\min(\alpha_1,\cdots,\alpha_n)>0$. If this holds true, then using the geometric decay in $k_1,k_2$, the last term will be bounded if one has
\begin{equation}\label{lemclaim2}
\sum_{Q\in\mathcal{Q}}|Q|\langle b_{1,q-k_1}\rangle_{Q,r}\langle b_{2,q-k_2}\rangle_{Q,s}\lesssim |Q_0|,
\end{equation}which follows from the disjointness of $\{b_{i,q-k_i}\chi_Q\}$ with fixed $k_i$, and $1/r+1/s\geq 1$. More precisely, one can first show that (\ref{lemclaim2}) holds true when $r=s=1$ and when $1/r+1/s=1$, then interpolate.

We are left with demonstrating (\ref{lemclaim}), which will be derived from the continuity inequality in Lemma \ref{ctyt}. We only prove the case $k_1\leq k_2$. The other case follows symmetrically as $A_Q$ is essentially self-adjoint ($A_\lambda$ is self-disjoint and the indicator function $\chi_{\frac{1}{3}Q}$ can be relaxed to $\chi_{Q}$ for the rest of the argument). Fix $Q\in\mathcal{Q}$ with $\ell(Q)=2^q$. The mean zero property of $b_{2,q-k_2}$ implies that
\[
\begin{split}
&\left|\left\langle A_Qb_{1,q-k_1},b_{2,q-k_2}\right\rangle\right|\\
\leq&\sum_{L\in\mathcal{L}:\,\ell(L)=2^{q-k_2}}\frac{1}{|L|}\left|\int_L\int_L \left(A_Qb_{1,q-k_1}(x)-A_Qb_{1,q-k_1}(x')\right) b_{2,q-k_2}(x)\,\d x\d x'\right|\\
\lesssim&\frac{1}{|L_0|}\int_{L_0}\left|\int \left(A_Q b_{1,q-k_1}(x)-\tau_yA_{Q}b_{1,q-k_1}(x) \right) b_{2,q-k_2}(x)\,\d x\right|\,\d y,
\end{split}
\]where $L_0$ is a $\gamma$-cube centered at the origin with $\ell(L_0)=2^{q-k_2+1}$. Let the (non-dyadic) $\gamma$-cube $\widetilde{Q}\supset Q$ be so that $\ell(\widetilde{Q})<2\ell(Q)$, applying Lemma \ref{ctyt} to $\widetilde{Q}$ implies that
\[
\begin{split}
\lesssim &\frac{1}{|L_0|}\int_{L_0} \left|\left(\frac{y_1}{2^{q\alpha_1}},\ldots,\frac{y_n}{2^{q\alpha_n}}\right)\right|^\eta |\widetilde{Q}| \langle b_{1,q-k_1}\chi_Q\rangle_{\widetilde{Q},r}\langle b_{2,q-k_2}\chi_Q\rangle_{\widetilde{Q},s}\,\d y\\
\lesssim &\left| \left(2^{-k_2\alpha_1},\ldots,2^{-k_2\alpha_n}\right)\right|^\eta |Q|\langle b_{1,q-k_1}\rangle_{Q,r}\langle b_{2,q-k_2}\rangle_{Q,s}\\
\lesssim &2^{-k_2\alpha\eta}|Q|\langle b_{1,q-k_1}\rangle_{Q,r}\langle b_{2,q-k_2}\rangle_{Q,s},
\end{split}
\]which completes the proof. 

\end{proof}

\section{Curves with nonvanishing torsion: proof of Theorem \ref{main2}}\label{nonvanishing}
The proof of Theorem \ref{main2} proceeds in a similar fashion as Theorem \ref{main}, so we only point out some key differences. 

First, it will again be helpful to study $H_\gamma$ by decomposing it into dyadic scales, and one can define the localized single scale operators in a similar way as before. Let $k\in\mathbb{Z}_-$, define the single scale operator $A_{2^k}^\gamma$ as in (\ref{average}). We decompose $A_{2^k}^\gamma$ into localized pieces and will again drop the dependence on $\gamma$ from now on. According to the Taylor expansion formula (\ref{Taylor}), there exists a constant $N=N(\gamma)$ sufficiently large such that for a dyadic cube $Q$ with $\ell(Q)=2^q$, $q\leq N$, the localized operator 
{{$$A_Qf:=A_{2^{q-N}}(f\chi_{\frac{1}{3}Q})$$}}is supported in $Q$. {{Here $cQ$ denotes the cube with the same center and orientation as $Q$ such that the side-length of $cQ$ of the side parallel to $e_k$ is $c\ell(Q)^k$, $k=1,\ldots,n$. Note that $cQ$ is not necessarily a $\gamma$-cube. Even though the coordinate system is rotated and the curve $\gamma$ is not exactly invariant under any nonisotropic dilations, the support condition still follows from the same reasoning as the monomial curve case observing the Taylor expansion (\ref{Taylor}) and the fact that vectors $\gamma^{(k)}(0)$ are uniformly bounded for $k=1,\ldots,n$, and then choosing the constant $N$ sufficiently large. }} 

By rotating the shifted dyadic grids constructed in (\ref{grids}), one obtains $3^n$ shifted grids such that the following decomposition holds true:
{{\[
A_{2^{q-N}}f=\sum_{\vec{j}}\sum_{\substack{Q:\,Q\in\mathcal{D}_{\vec{j}}\\ \ell(Q)=2^{q}}} A_Qf.
\]}}There is thus the decomposition
\[
H_\gamma= \sum_{\vec{j}}\sum_{Q\in\mathcal{D}_{\vec{j}}}A_Qf,
\]and it suffices to study the operator $\sum_{Q\in\mathcal{D}}A_Qf$ for a fixed dyadic grid $\mathcal{D}$.

The proof of Theorem \ref{main2} would be complete once we verify that there holds a version of Lemma \ref{ctyt}, the $L^p\to L^q$ improving continuity estimate for the curves of nonvanishing torsion as well, which we state as the lemma below. Note that the rest of the proof proceeds in the same (in fact even more straightforward) way as the one for Theorem \ref{main}, as the cubes here have the same nonisotropic dilation structure as the ones associated to the moment curve.
\begin{remark}\label{small cube}
One can construct the sparse collection $\mathcal{S}$ so that all $\gamma$-cubes $Q\in\mathcal{S}$ satisfy $\ell(Q)\lesssim_N 1$. This can be done by first identifying a disjoint collection of $\gamma$-cubes $\{Q_0^1,\ldots,Q_0^M\}$ whose sidelength $\ell(Q_0^i)=2^{N}$, $\forall 1\leq i\leq M$, such that the union of $\left\{\frac{1}{3}Q_0^i\right\}$ covers the support of $f$. Then decompose $f$ into $M$ parts each of which is supported in $\frac{1}{3}Q_0^i$ for some $i$. One can then proceed similarly as in the proof of Theorem \ref{main} to construct a sparse collection for each part. The union of the $M$ sparse collections obviously satisfies the desired property.
\end{remark}
\begin{lemma}\label{cty2}
Let $f_1,f_2$ be supported in a $\gamma$-cube $Q$ and $\lambda\sim \ell(Q)$. For any $y\in\mathbb{R}^n$ such that $|y_j|\leq\ell(Q)^{j}$, $\forall j=1,\ldots,n$, and $(r,s)$ in the {{range as in the statement of Theorem \ref{main2}}}, there holds for some constant $\eta=\eta(r,s,n)$ that
\[
|\langle A_{\lambda}^\gamma f_1-\tau_y A_{\lambda}^\gamma f_1,f_2\rangle|\lesssim \left|\left(\frac{y_1}{\ell(Q)},\ldots,\frac{y_n}{\ell(Q)^{n}}\right)\right|^\eta |Q|\langle f_1\rangle_{Q,r}\langle f_2\rangle_{Q,s},
\]where the implicit constant depends on $r,s,\gamma$.
\end{lemma}

\begin{proof}
Similarly as in the monomial curve case, it suffices to show that
\begin{equation}\label{last one}
\|A_\lambda^\gamma-\tau_yA_{\lambda}^\gamma:\, L^r\to L^{s'}\|\lesssim \left|\left(\frac{y_1}{\lambda},\ldots,\frac{y_n}{\lambda^n}\right)\right|^\eta\lambda^{(1/s'-1/r)\sum_{j=1}^n j}
\end{equation}whenever $|y_j|\leq \lambda^{j}$, $\forall j=1,\ldots,n$.

Since the curve $\gamma$ has nonvanishing torsion everywhere, for any pair $(r,s)$ in the region dual to the claimed range, there holds the $L^p\to L^q$ improving estimate for the associated single scale operator $A_1^\gamma$ according to \cite{TW}. Then similarly as in Lemma \ref{ctyt}, the Fourier decay implies the desired estimate for $\lambda=1$:
\begin{equation}\label{A1}
\|A_1^\gamma-\tau_y A_1^\gamma: L^r\to L^{s'}\|\lesssim |y|^\eta,\,\quad |y_j|\leq 1,\,\forall j=1,\ldots,n,
\end{equation}for some $\eta=\eta(r,s,n)>0$. We also remark that the $|y_j|\leq 1$ condition is not crucial in the above as the $|y_j|>1$ case is trivial. And it can be easily seen that (\ref{A1}) holds true for any $\lambda$ as well with the implicit constant on the RHS depending on $\lambda$. Therefore it is left to prove the case at scale $\lambda\sim 2^q$, for $q\leq N$ larger than a fixed constant $M_1$ to be determined later.

Fix a scale $\lambda\sim\ell(Q)=2^q$, $q\leq N$, we will do a change of variables to dilate the curve at scale $2^q$ localized at $Q$ to another curve at scale $1$, and apply the above estimate. What is important is to make sure that the dilated curve still has nonvanishing torsion everywhere. This is indeed the case. To see this, we first decompose $A_{\lambda}^\gamma$ into smaller pieces:
\[
A_{\lambda}^\gamma f(x)=\sum_{i=1}^{M_2}A_{\lambda,i}^\gamma f(x):=\sum_{i=1}^{M_2} \int_{I_i}f(x-\gamma(|t|))\frac{\d t}{t},\quad I_i:=\left[\frac{\lambda}{2}+\frac{i-1}{M}\cdot\frac{\lambda}{2}, \frac{\lambda}{2}+\frac{i}{M_2}\cdot\frac{\lambda}{2}\right],
\]where $M_2$ is a sufficiently large constant to be chosen later. It thus suffices to prove (\ref{last one}) for each $i$ by triangle inequality.

Fix $i=1,\ldots,M_2$, construct a new coordinate system determined by the derivatives of $\gamma$ at the left endpoint $a_i$ of $I_i$ instead of the origin. We refer to the Taylor expansion fomular (\ref{Taylor}) for such a construction. Let $\gamma_k(t)$ be the $k$-th component of $\gamma(t)$ in the new coordinate system, $\forall k=1,\ldots,n$. Then the construction of the coordinate system implies that 
\[
\gamma_k^{(\ell)}(a_i)=0,\quad \forall 1\leq\ell<k\leq n.
\]For any $\epsilon>0$, there exists $M_2$ sufficiently large so that for any $t\in I_i$,
\[
\gamma_k^{(\ell)}(t)<\epsilon,\quad \forall 1\leq\ell<k\leq n,
\]which implies that
\[
0\neq\det\begin{pmatrix} \gamma_1^{(1)}(t) & \gamma_2^{(1)}(t)&\cdots &\gamma_n^{(1)}(t)\\
\gamma_1^{(2)}(t) &\gamma_2^{(2)}(t) & \cdots &\gamma_n^{(2)}(t)\\
\vdots & \vdots &\ddots &\vdots\\
\gamma_1^{(n)}(t)& \gamma_2^{(n)}(t)& \cdots &\gamma_n^{(n)}(t)
\end{pmatrix}\approx \prod_{k=1}^n\gamma_k^{(k)}(t),
\]where the approximation of the determinant by its trace follows from the fact that the matrix is essentially lower triangular.

Now consider the dilation $\delta_\lambda (x):=(\lambda x_1,\ldots,\lambda^n x_n)$ in the fixed coordinate system determined by $a_i$. For any function $f$,
\[
A_1^\nu f(x):=A^\gamma_{\lambda,i}(f\circ\delta^\gamma_{\lambda^{-1}}(\cdot))(\delta^\gamma_{\lambda}(x))
\]is a single scale convolution operator associated to a new curve $\nu$ at scale $1$. In particular, $\nu$ satisfies for any $t$ so that $\lambda t\in I_i$ that
\[
\nu_k(t)=\lambda^{-k}\gamma_k(\lambda t),\quad k=1,\ldots,n.
\]Therefore, it suffices to verify that $\nu$ has nonvanishing torsion everywhere, which combined with (\ref{A1}) and change of variables will complete the proof.

At any $t$ such that $\lambda t\in I_i$, the torsion of $\nu$ is equal to
\[
\begin{split}
\det\begin{pmatrix} \nu^{(1)}(t) & \cdots & \nu^{(n)}(t)\end{pmatrix}=&\det\begin{pmatrix} \nu_1^{(1)}(t) &\nu_1^{(2)}(t)&\cdots &\nu_1^{(n)}(t)\\
\nu_2^{(1)}(t) &\nu_2^{(2)}(t) &\cdots &\nu_2^{(n)}(t)\\
\vdots & \vdots &\ddots &\vdots\\
\nu_n^{(1)}(t) &\nu_n^{(2)}(t) &\cdots &\nu_n^{(n)}(t)
\end{pmatrix}\\
=&\det\begin{pmatrix}  \gamma_1^{(1)}(\lambda t) & \lambda\gamma_1^{(2)}(\gamma t) &\cdots & \lambda^{n-1}\gamma_1^{(n)}(\lambda t)\\
\lambda^{-1}\gamma_2^{(1)}(\lambda t) & \gamma_2^{(2)}(\lambda t) &\cdots & \lambda^{n-2}\gamma_2^{(n)}(\lambda t)\\
\vdots &\vdots &\ddots &\vdots\\
\lambda^{-(n-1)}\gamma_n^{(1)}(\lambda t) &\lambda^{-(n-2)}\gamma_n^{(2)}(\lambda t) &\cdots &\gamma_n^{(n)}(\lambda t)
\end{pmatrix}\\
\approx &\prod_{k=1}^n \gamma_k^{(k)}(\lambda t)\neq 0,
\end{split}
\]where the approximation on the last line follows from the fact all the upper triangular entries in the matrix are negligibly small when $\lambda$ is sufficiently small, which can be achieved by fixing the constant $M_1$ large enough. Therefore, the curve $\nu$ has nonvanishing torsion and the proof is complete.
\end{proof}

\section{Weighted estimates and sharpness of sparse bounds}\label{weighted}
\subsection{Weighted norm inequalities}
Sparse domination is particularly powerful in deducing sharp weighted norm inequalities. In this section, we introduce such results as corollaries of Theorem \ref{main} and \ref{main2}. The class of weights that are involved are the Muckenhoupt $A_p$ weights defined via $\gamma$-cubes, which can be described as special cases of weights in spaces of homogeneous type. These seem to be the first weighted norm inequalities for singular integrals along curves, and the estimates are sharp for the sparse forms.

Given a set $X$, the triple $(X,\rho,\mu)$ is called a \emph{space of homogenous type} if $\rho$ is a quasi-metric, that is, a metric except that the triangle inequality is replaced by the quasi-triangle inequality:
\[
\rho(x,z)\leq C_0(\rho(x,y)+\rho(y,z))\qquad \forall x,y,z\in X,
\]and the positive measure $\mu$ is doubling:
\[
0<\mu(B(x,2r))\leq C_1\mu(B(x,r))<\infty,\qquad \forall x\in X, r>0,
\]for some absolute constants $C_0,C_1>0$. Such spaces were first introduced by Coifman and Weiss in \cite{CW}. It is a natural extension of the Euclidean space $\mathbb{R}^n$ equipped with Lebesgue measure, and many results in the Euclidean space extend to this setting. We also refer the reader to \cite{Christ2} for a construction of dyadic cubes in general spaces of homogeneous type.

Given a space of homogeneous type $(X,\rho,\mu)$ and $1<p<\infty$, there is an associated class of Muckenhoupt type $A_p$ weights consisting of locally integrable positive functions $w$ on $X$ such that
\[
[w]_{A_{p}}:=\sup_{B}\langle w\rangle_{B}\langle w^{1-p'}\rangle_B^{p-1}<\infty,
\]where the supremum is taken over all quasi-balls $B$ in $(X,\rho,\mu)$. We will also need $RH_p$, the reverse H\"older class of weights, containing locally integrable positive functions $w$ such that
\[
[w]_{RH_p}:=\sup_{B}\langle w\rangle_B^{-1}\langle w\rangle_{B,p}<\infty,
\]where again $B$ ranges over all quasi-balls in $(X,\rho,\mu)$. The following estimate is obtained in Section 6 of \cite{BFP}.
\begin{lemma}\label{sparseweight}
Given $r,s\in [1,\infty)$ and a sparse collection $\mathcal{S}$ (w.r.t. measure $\mu$) of quasi-balls $B$ in $(X,\rho,\mu)$. Then for any $r<p<s'$ there exists constant $C=C(r,s,p)$ so that for every weight $w\in A_{\frac{p}{r}}\cap RH_{(\frac{s'}{p})'}$, there holds
\[
\Lambda_{\mathcal{S},r,s}(f,g)\leq C\left([w]_{A_{\frac{p}{r}}}[w]_{RH_{(\frac{s'}{p})'}}\right)^\alpha\|f\|_{L^p(w)}\|g\|_{L^{p'}(w^{1-p'})},\quad \alpha:=\max\left(\frac{1}{p-r},\frac{s'-1}{s'-p}\right).
\]
\end{lemma}
Lemma \ref{sparseweight}, together with Theorem \ref{main} and \ref{main2}, implies immediately the following weighted norm inequalities for Hilbert transforms along curves.
\begin{corollary}\label{weight}
Let $\gamma$ be a monomial curve (resp. curve with nonvanishing torsion) and $(r,s)$ be in the range as in the statement of Theorem \ref{main} (resp. Theorem \ref{main2}). Then for any $r<p<s'$, and weight $w\in A_{\frac{p}{r}}\cap RH_{(\frac{s'}{p})'}$ defined via $\gamma$-cubes as in Definition \ref{gcubedef} (resp. Definition \ref{cubenvt}), there holds
\[
\|H_{\gamma}\|_{L^p(w)\to L^p(w)}\leq C\left([w]_{A_{\frac{p}{r}}}[w]_{RH_{(\frac{s'}{p})'}}\right)^\alpha,\quad \alpha:=\max\left(\frac{1}{p-r},\frac{s'-1}{s'-p}\right).
\]
\end{corollary}
In the above, when $\gamma$ is taken as a general curve with nonvanishing torsion, the Hilbert transform $H_{\gamma}$ is understood as the local operator defined in (\ref{Hnvt}).

It suffices to verify that, for $\gamma$ being a monomial curve or a $C^n$ curve with nonvanishing torsion, there exists a space of homogeneous type $(X,\rho,\mu)$ such that the associated $\gamma$-cubes are comparable to the quasi-balls. Indeed, let $X=\mathbb{R}^n$ and $\mu$ be the Lebesgue measure on $\mathbb{R}^n$. 

If $\gamma(t)$ is a monomial curve with parameters $(\alpha_1,\ldots,\alpha_n)$, this can be seen by defining a quasi-metric
\[
\rho_\gamma(x,y):=\max_{j=1,\ldots,n}\left\{|x_j-y_j|^{1/\alpha_j}\right\},\qquad \forall x,y\in X,
\]which turns $(X,\rho_\gamma,\mu)$ into a space of homogeneous type with constants $C_0,C_1$ depending on $\gamma$ and the dimension.

If $\gamma$ is a $C^n$ curve with nonvanishing torsion, rotate the coordinate system according to the orientation of $\gamma$-cubes and define
\[
\rho_{\gamma}(x,y):=\max_{j=1,\ldots,n}\left\{|x_j-y_j|^{1/j}\right\},\qquad \forall x,y\in X,
\]where $x_j$ denotes the $j$-th component of $x$ in the new coordinate system. Then again, this obviously gives rise to a space of homogeneous type. Note that in this case, one can in fact slightly improve the weighted estimates for $H_\gamma$ by considering the class of $A_p$ weights defined via $\gamma$-cubes with side-length $\ell\lesssim 1$. These cubes are comparable to the quasi-balls defined via $\rho_{\gamma}$ and determine a larger class of weights. We omit the details.

\subsection{Sharpness of the sparse region}
The region $\Omega$ of $(r,s)$ where Theorem \ref{main} and \ref{main2} hold true is sharp up to the endpoint, which follows from the fact that the $L^p$ improving estimate for single scale operators $A_\lambda^\gamma$ fails to hold outside the region $\Omega$.

More precisely, fix a curve $\gamma$ in either of the two classes and a pair $(r,s)$ in the exterior of the region $\Omega$ as in the statement of Theorem \ref{main} and \ref{main2}. Suppose for any functions $f,g$ there exists a sparse collection $\mathcal{S}$ of $\gamma$-cubes such that
\[
\langle H_\gamma f,g\rangle \leq C\Lambda_{\mathcal{S},r,s}(f,g).
\]

Let $A_\lambda^\gamma$ be the single scale operator defined in (\ref{average}) and
\[
\widetilde{A}_\lambda^\gamma h(x):=\int_{\frac{\lambda}{2}\leq t<\lambda}h(x-\gamma(t))\,dt
\]be the averaging operator of the same scale (for $t$ positive). Without loss of generality, one can assume that the function $h$ is nonnegative. It is justified in \cite{TW} that for any small neighborhood of $0$, say, $B_\epsilon(0)$ for $\epsilon$ to be determined later, there exists a sequence of nonnegative functions $\{f_k\}$ supported in $B_{\epsilon}(0)$ such that $\|f_k\|_{L^r}=1$, $\forall k$, while
\[
\|\widetilde{A}^{\gamma}_1f_k\|_{L^{s'}}\to \infty.
\]Since the support of $\{f_k\}$ is small, the positive and negative halves of $A^\gamma_1f_k$ (the part of the integral over positive and negative $t$ respectively) are supported disjointly from each other, hence
\[
\|A^\gamma_1 f_k\|_{L^{s'}}\gtrsim \|\widetilde{A}^{\gamma}_1f_k\|_{L^{s'}}\to \infty.
\]

Therefore, there exists a sequence of nonnegative functions $\{g_k\}$ such that $\|g_k\|_{L^s}=1$, $\forall k$, while
\[
\langle A^\gamma_1 f_k,g_k\rangle=\|A^\gamma_1 f_k\|_{L^{s'}}\to \infty.
\]One can obviously assume that $g_k$ is supported in an $\epsilon$-neighborhood of $\left\{\gamma(t):\,t\in\left[\frac{1}{2},1\right]\right\}$, which implies that
\[
\Lambda_{\mathcal{S}(f_k,g_k),r,s}(f_k,g_k)\gtrsim\langle H_\gamma f_k,g_k\rangle=\sum_{j=-1}^1 \langle A^\gamma_{2^j} f_k,g_k\rangle\geq \langle A^\gamma_1 f_k,g_k\rangle\to \infty.\]The last inequality above follows from the fact that all three forms in the intermediate step are nonnegative due to the support consideration of $g_k$. However, observe that
\[
\text{Dist}\,(\text{spt}\,f_k,\text{spt}\,g_k)\gtrsim 1,
\]hence the sparse form satisfies
\[
\Lambda_{\mathcal{S}(f_k,g_k),r,s}(f_k,g_k)\lesssim \|f_k\|_{L^r}\|g_k\|_{L^s}=1,
\]which is a contradiction.




\bibliography{HTcurves}

\providecommand{\bysame}{\leavevmode\hbox to3em{\hrulefill}\thinspace}
\providecommand{\MR}{\relax\ifhmode\unskip\space\fi MR }
\providecommand{\MRhref}[2]{%
  \href{http://www.ams.org/mathscinet-getitem?mr=#1}{#2}
}
\providecommand{\href}[2]{#2}
\begin{thebibliography}{10}

\bibitem{BFP}
Fr\'ed\'eric Bernicot, Dorothee Frey, and Stefanie Petermichl, \emph{Sharp
  weighted norm estimates beyond {C}alder{\'o}n-{Z}ygmund theory}, Anal. PDE
  \textbf{9} (2016), no.~5, 1079--1113.

\bibitem{Christ1}
M.~Christ, \emph{Convolution, curvature, and combinatorics: A case study}, Int.
  Math. Res. Not. \textbf{19} (1988), 1033--1048.

\bibitem{Christ2}
\bysame, \emph{A ${T}(b)$ theorem with remarks on analytic capacity and the
  {C}auchy integral}, Colloquium Mathematicae \textbf{60-61} (1990), no.~2,
  601--628.

\bibitem{CW}
R.~R. Coifman and G.~Weiss, \emph{Analyse harmonique non-commutative sur
  certains espaces homog\`enes}, Lecture Notes in Math. 242, Springer, Berlin,
  1971.

\bibitem{CDPO}
Amalia Culiuc, Francesco Di~Plinio, and Yumeng Ou, \emph{Uniform sparse
  domination of singular integrals via dyadic shifts}, preprint
  arXiv:1610.01958, to appear, Math. Res. Lett. (2016).

\bibitem{Lacey1}
Michael~T.\ Lacey, \emph{An elementary proof of the ${A}_2$ {B}ound}, preprint
  arXiv:1501.05818, to appear, Israel J.\ Math. (2015).

\bibitem{Lacey2}
Michael~T. Lacey, \emph{Sparse bounds for spherical maximal functions},
  preprint arXiv: 1702.08594 (2017).

\bibitem{LM}
Michael~T.\ Lacey and Dar{\`\i}o Mena, \emph{The sparse ${T}(1)$ theorem},
  preprint arXiv:1610.01531 (2016).

\bibitem{Lerner}
Andrei~K. Lerner, \emph{On pointwise estimates involving sparse operators}, New
  York J. Math. \textbf{22} (2016), 341--349. \MR{3484688}

\bibitem{Oberlin}
R.~Oberlin, \emph{Sparse bounds for a prototypical singular {R}adon transform},
  preprint arXiv:1704.04297 (2017).

\bibitem{SW}
E.~M. Stein and S.~Wainger, \emph{Problems in harmonic analysis related to
  curvature}, Bull. Amer. Math. Soc. \textbf{84} (1978), 1239--1295.

\bibitem{TW}
T.~Tao and J.~Wright, \emph{${L}^p$ improving bounds for averages along
  curves}, J. Amer. Math. Soc. \textbf{16} (2003), no.~3, 605--638.

\end{thebibliography}
%
\bibliographystyle{amsplain}
\end{document}